\documentclass[11pt]{amsart}

  \usepackage{enumerate}

  \usepackage{amssymb}

  \advance\hoffset by -0.5 truecm

  \def \dim{{\mbox {dim}}\,}

 \def \ga{\gamma}
 \def \Ga{\Gamma}

 \def \e{\varepsilon}
 \def \la{\lambda}
 
 \def \vr{\varphi}
 
 \def \Si{\Sigma}
 \def \th{\theta}

 \def \re{{\mathbb R}}
 \def \Rm{{\mathbb R}}
 \def \na{{\mathbb N}}
 
 \def \Z {{\mathbb Z}}

 \def \T {{\mathbb T}}
 \def \TTM{{\T\times TM}}
 \def \TM{{\T\times M}} 
 \def \dga{\dot\gamma}

 \def \lv{\left\vert}
 \def \rv{\right\vert}
 \def \lV{\left\Vert}
 \def \rV{\right\Vert}
 \def \ov{\overline}
 
 \def \then{\Longrightarrow}

 \def \lan{\langle}
 \def \ran{\rangle}
 \def \lto{\longrightarrow}
 \def \lmto{\longmapsto}
 \def \fa{\forall\;} 

 \DeclareMathOperator*\argmin{arg\, min}   
 \DeclareMathOperator\diam{diam}   
 \DeclareMathOperator\supp{supp}   

 \def \cH{{\mathcal H}} 
 \def \cO{{\mathcal O}} 
 \def \cP{{\mathcal P}}
 
 \def \fM{{\mathfrak  M}}
 
 \def \ocA{{\ov{\mathcal A}}}

 \theoremstyle{plain}

 \newtheorem{Thm}{\bf Theorem}

 \newtheorem{Lemma}[Thm]{\bf Lemma}
 \newtheorem{lem}[Thm]{\bf Lemma}
 \newtheorem{Cor}[Thm]{\bf Corollary}
 \newtheorem{Corollary}[Thm]{\bf Corollary}
 
 \newtheorem{Proposition}[Thm]{\bf Proposition}
 \newtheorem{prop}[Thm]{\bf Proposition}
 
\def\le{\leqslant}
\def\ge{\geqslant}

 \theoremstyle{definition}

 \title[]
 {A generic property of families of Lagrangian systems}
 \author{Patrick Bernard}
 \address{CEREMADE\\
          Universit\'e de Paris Dauphine\\
          Pl. du Mar\'echal de Lattre de Tassigny\\
          75775 Paris Cedex 16\\
          France\\
	  }
 \email{patrick.bernard@ceremade.dauphine.fr}

 \author{Gonzalo Contreras}
 \thanks{Gonzalo Contreras was partially supported by 
 CONACYT-M\'exico grant \# 46467-F}
   \address{CIMAT\\
           P.O. Box 402\\
           36.000 Guanajuato, GTO \\
           M\'exico.}
  \email{gonzalo@cimat.mx}

\begin{document}

 \begin{abstract}
 
  We prove that a generic lagrangian has finitely many
  minimizing measures for every cohomology class. 
 \end{abstract}

  \maketitle
  Annals of Maths, 167 No3. (2008).
  \section{Introduction}
  Let $M$ be a compact boundaryless smooth manifold. 
  
  Let $\T$ be either the group $(\re/\Z,+)$ or the trivial group
  $(\{0\},+)$. 
  
  A {\it Tonelli Lagrangian} is a $C^2$ function
  $L:\TTM\to\re$ such that 
  \begin{itemize}
  \item The restriction to each fiber of $\TTM\to\TM$ is a  {\it
  convex} function.
   \item It is fiberwise {\it superlinear}:
      $$
      \lim_{|\th|\to+\infty}{L(t,\th)}/{|\th|}=+\infty,\qquad
      (t,\th)\in\TTM.
      $$	
   \item The Euler-Lagrange equation
%    $$
%    \frac{d}{dt}\,\frac{\partial L}{\partial v}=
%    \frac{\partial L}{\partial x}
%    $$   
   $$
   \tfrac{d}{dt}L_v=L_x
   $$
   defines a {\it complete} flow $\vr:\re\times(\TTM)\lto\TTM$.   
  \end{itemize}
 We say that a Tonelli Lagrangian $L$ is \emph{strong Tonelli} if 
 $L+u$ is a Tonelli Lagrangian for each $u\in C^{\infty}(\TM,\re)$.
 When $\T=\{0\}$ we say that the lagrangian is {\it autonomous}.
  
  Let $\cP(L)$ be the set of Borel probability measures on $\TTM$
  which are invariant under the Euler-Lagrange flow $\vr$. 
  The action functional $A_L:\cP(L)\to\re\cup\{+\infty\}$ is defined
  as
  $$
  A_L(\mu):=\lan L,\mu\ran:=\int_{\TTM}L\;d\mu.
  $$
  The functional $A_L$ is lower semi-continuous and the 
  minimizers of $A_L$ on $\cP(L)$ are called {\it minimizing
  measures}.
  %correction
  % It follows (what does it follow from?)
  The ergodic components of a minimizing
  measure are also minimizing, and they
  are mutually singular,  so that
  the  set $\fM(L)$ of minimizing  measures is 
  a simplex whose extremal points are the ergodic minimizing measures.
  
  In general, the simplex $\fM(L)$ may be of infinite dimension. 
  The goal of the present paper is to prove that this is a very
 exceptional phenomenon.
 The first  results in that direction 
 %(and, up to now, the only ones)
 were obtained by Ma\~n\'e in \cite{Ma6}. 
 His paper has been very influential to our work.

  We say that a property is {\it generic} in the 
 sense of Ma\~n\'e 
if, for each  strong Tonelli Lagrangian $L$, there exists
a residual subset  $\cO\subset C^{\infty}(\TM,\re)$
such that the property holds for all the Lagrangians
$L-u,u\in \cO$. 
A set is called residual if it is a countable 
intersection of open and dense sets.
We recall which topology is used on 
$C^{\infty}(\TM,\re)$.
Denoting by $\|u\|_k$ the $C^k$-norm of a function 
$u:\TM\lto \Rm$,
define
$$
\|u\|_{\infty}:=\sum_{k\in \na}
\frac{\arctan (\|u\|_k)}{2^k}.
$$
Note that $\| . \|_{\infty}$
is not a norm.
Endow the space $C^{\infty}(\TM,\re)$
with the translation-invariant metric 
$\|u-v\|_{\infty}$.
This metric is complete, hence the Baire property holds:
any residual subset of $C^{\infty}(\TM,\re)$ is  dense.

\begin{Thm}\label{main}
Let $A$ be a finite dimensional convex family
of strong Tonelli Lagrangians. 
Then there exists a residual subset $\cO$
of $C^{\infty}(\TM,\re)$ such that, 
$$
u\in\cO,\quad L\in A \quad\then\quad
\dim\fM(L-u)\le \dim A.
$$
In other words, 
 there exist at most
$1+\dim A$ ergodic  minimizing measures of
$L-u$.
\end{Thm}

The main result of Ma\~n\'e in  \cite{Ma6}
is that  having 
a unique  minimizing measure is a 
generic property.
This corresponds to the case where $A$ is a point
in our statement.
Our generalization of Ma\~n\'e's result is motivated
by the following construction due to John Mather:

We can view a 1-form on $M$ as a function on $TM$ which 
is linear on the fibers. If
$\lambda$ is closed, 
the Euler-Lagrange equation
of the Lagrangian $L-\lambda$ is the same as that of $L$.
However, the minimizing measures of $L-\lambda$,
are not the same as the minimizing measures of $L$.
Mather proves in~\cite{Mat5} that the set $\fM(L-\la)$
of minimizing measures of the lagrangian $L-\la$ depends only
on the cohomology class $c$ of $\la$. If $c\in H^1(M,\re)$ we write
$\fM(L-c):=\fM(L-\la)$, where $\la$ is a closed form of cohomology
$c$.
%We call
%$$
%\fN(L):=\bigcup_{c\in H^1(M,\re)}\fM(L-c)
%$$
%the set of all {\it Mather minimizing measures} of $L$.

It turns out that important applications of Mather theory, such as
the existence of  orbits wandering in phase space, require
understanding  not only of the set $\fM(L)$ of minimizing measures 
for a fixed or generic cohomology classes but of the set 
of all Mather minimizing measures for every $c\in H^1(M,L)$.
The following corollaries are crucial for  these
applications.

\begin{Cor}
The following property is generic in the sense of Ma\~ n\'e:

For all $c\in H^1(M,\Rm)$, there are at most
$1+\dim H^1(M,\Rm)$ ergodic minimizing measures of
$L-c$.
\end{Cor}

We say that a property is of infinite codimension
if, for each finite dimensional convex family  $A$ of 
strong Tonelli Lagrangians,
there exists  a residual subset $\cO$
in $C^{\infty}(\T\times M,\re)$ such that
none of the Lagrangians $L-u$, $L\in A$, $u\in \cO$
 satisfy the property.

\pagebreak

\begin{Cor}
The following property is of infinite codimension:

There exists $c\in H^1(M,\Rm)$,
such that
$L-c$ has infinitely many ergodic minimizing measures.
\end{Cor}

Another important  issue concerning variational methods for 
Arnold diffusion questions is the total disconnectedness of
the quotient Aubry set. John Mather proves in~\cite[\S~3]{Mat10}
that the quotient Aubry set $\ocA$ of any Tonelli lagrangian on $\TTM$ 
with $\T=\re/\Z$ and $\dim M\le 2$  (or with $\T=\{0\}$ and 
$\dim M\le 3$) is totally disconnected. See~\cite{Mat10} for its definition. 

The elements of the quotient Aubry set are called {\it static
classes}. They are disjoint subsets of $\TTM$ and each static class
supports at least one ergodic minimizing measure. We then get

\begin{Cor}
The following property is generic in the sense of Ma\~ n\'e:

%\noindent
For all $c\in H^1(M,\re)$  the quotient Aubry set
$\ocA_c$ of $L-c$ has at most $1+\dim H^1(M,\re)$ elements.
\end{Cor}

  \section{Abstract Results}
  
   Assume that we
  are given 
  \begin{itemize}
  \item Three topological vector spaces $E$, $F$, $G$.
  \item A continuous linear map $\pi:F\to G$.
  \item A bilinear pairing $\langle u,\nu\rangle:E\times G\to\re$.
%correction
  \item Two metrizable convex  compact subsets $H\subset F$ and $K\subset G$
  such that $\pi(H)\subset K$.
  \end{itemize}
  
  Suppose that
  \begin{enumerate}
  \renewcommand\theenumi{\roman{enumi}}
  \item The map 
   $$
   E\times K\ni (u,\nu)\longmapsto \lan u,\nu\ran
   $$
  is continuous.
  
  \medskip

  We will also denote $\lan u,\pi(\mu)\ran$ by $\lan u,\mu\ran$
  when $\mu\in H$. Observe that each element $u\in E$ gives rise to a
  linear functional on $F$
  $$
  F \ni \mu\longmapsto\lan u,\mu\ran
  $$
  which is continuous on $H$. We shall denote by $H^*$ the set of
  affine and continuous functions  on $H$ and use the
  same symbol $u$ for an element of $E$ and for the element
  $\mu\longmapsto\lan u,\mu\ran$ of $H^*$ which is associated to it.
  
  \bigskip
  
  \item The compact $K$ is separated by $E$. This means that,
        if $\eta$ and $\nu$ are two different points of $K$,
        then there exists a point $u$ in  $E$
        such that $\langle u,\eta\rangle \neq \langle u,\nu\rangle $.

    \medskip

         Note that the topology on $K$ is then the
         weak topology associated to $E$. 
	 A sequence $\eta_n$ of elements of $K$ 
	 converges to $\eta$ if and only if we have
         $
         \langle u,\eta_n\rangle\lto \langle u,\eta\rangle
         $
         for each $u \in E$. We shall, for notational conveniences,
         fix once and for all a metric $d$ on $K$.

       \bigskip

 \item $E$ is a Frechet space. It means that $E$ is a topological 
       vector space whose topology is defined by a 
       translation-invariant metric, and that $E$ is complete 
       for this metric.
       
       \medskip
       
       Note then that $E$ has the Baire property.
       We say that a subset is residual
       if it is a countable intersection of open and dense sets.
       The Baire property says that any residual subset of $E$ 
       is dense.

  \end{enumerate}

      Given $L\in H^*$ denote by
      $$
      M_H(L) := \argmin_H L
      $$
      the set of points $\mu\in H$ which minimize $L|_H$, 
      and by $M_K(L)$ the image $\pi (M_H(L))$.
      These are compact convex  subsets of $H$ and $K$.

Our main abstract result is:

\begin{Thm}\label{abstract}
For every finite dimensional affine subspace  $A$ of $H^*$,
there exists a residual subset $\cO(A)\subset E$
such that, for all $u\in \cO(A)$ and all $L\in A$,
we have
\begin{equation}\label{dim}
\dim M_K(L-u) \leq \dim A.
\end{equation}
\end{Thm}

 \begin{proof}
 
 We define the $\e$-neighborhood $V_{\e}$ of a
subset $V$ of $K$ as the union of all the open balls in $K$ 
which have radius $\e$ and are centered in $V$.
Given a subset $D\subset A$, a positive number $\e$,
and a positive integer $k$,  denote by
$\cO(D,\e,k)\subset E$
the set of points $u\in E$
such that, for each $L\in D$, the convex set
$M_K(L-u)$ is contained in the $\e$-neighborhood
of some $k$-dimensional convex subset of $K$.

We shall prove that the theorem holds with
$$
\cO(A)=\bigcap_{\e>0} \cO(A,\e,\dim A).
$$
If $u$ belongs to $\cO(A)$, then
(\ref{dim}) holds  for every $L\in A$.
Otherwise, for some $L\in A$, the convex set  $M_K(L-u) $
would contain a ball of dimension $\dim A+1$,
 and, if $\e$ is small enough, such a ball is not
 contained in the $\e$-neighborhood
 of any convex set of dimension $\dim A$.

So we have to prove that $\cO(A)$ is residual.
In view of the Baire property, it is enough to
check that, for any compact subset $D\subset A$
and any positive $\e$,
the set $\cO(D,\e,\dim A)$
is open and dense.
We shall prove in \ref{open} that it is open,
 and in \ref{dense} that it is dense.

 \end{proof}

 \subsection{Open}\label{open}\quad
 
 We prove that, for any $k\in\Z^+$, $\e>0$ and any 
 compact $D\subset A$, the set $\cO(D,\e,k)\subset E$
 is open. We need a Lemma.

\begin{lem}
\label{semicontinuity}
The set-valued map $(L,u)\lmto  M_H(L-u)$
is upper semi-continuous on $A\times E$.
This means that for any open subset $U$ of $H$, the set 
$$
\{(L,u)\in A\times E \;:\;
M_H(L-u)\subset U\}
\subset A\times E
$$
is open in $A\times E$.
Consequently, the set-valued map 
$(L,u)\lmto  M_K(L-u)$
is also upper semi-continuous.
\end{lem}

\begin{proof}
This is a standard consequence of the continuity of the map
$$
A\times E\times H\ni(L,u,\mu)\lmto
(L-u)(\mu)=
L(\mu)-\langle u,\mu\rangle.
$$
\end{proof}

Now let $u_0$ be a point  of $\cO(D,\e,k)$.
For each $L\in D$, there exists a $k$-dimensional
convex set $V\subset K$ such that 
$
M_K(L-u_0)\subset V_{\e}.
$
In other words, the open sets of the form
$$
\{(L,u)\in D\times E \;:\;
M_H(L-u)\subset V_{\e}\}
\subset D\times E,
$$
where $V$ is some $k$-dimensional convex subset of $K$,
cover the compact set $D\times\{u_0\}$.
So there exists a finite subcovering of $D\times\{u_0\}$
by open sets of the form $\Omega_i\times U_i$,
where $\Omega_i$ is an open set in $A$ and 
$U_i\subset \cO(\Omega_i,\e,k)$ is an open set in $E$
containing $u_0$.
We conclude that 
the open set $\cap U_i$ is contained in $\cO(D,\e,k)$,
and contains $u_0$. This ends the proof.

\qed

 \subsection{Dense}\label{dense}\quad
 
 We prove the density of $\cO(A,\e,\dim A)$ in $E$
for $\e>0$.
Let $w$ be a point in $E$. We want to prove that 
$w$ is in the closure of 
$\cO(A,\e,\dim A)$.

\begin{lem}\label{map}
There exists an integer $m$ and a continuous map
$$
T_m=(w_1,\ldots,w_m):K\lto \Rm^m,
$$
with $w_i\in E$
such that 
\begin{equation}\label{diam}
\fa x\in\re^m \qquad \diam T_m^{-1}(x) <\e,
\end{equation}
where the diameter is taken for the distance $d$ on $K$.

% such that the preimage $T_m^{-1}(x)$
% by $T_m$ of each point $x\in \Rm^m$
% has diameter less than $\e$ 
% for the distance $d$ on $K$. 
\end{lem}
\begin{proof}
In $K\times K$,
to each element $w\in E$ we associate the open set
$$
U_w=\{\,(\eta,\mu)\in K\times K \; :\;
\langle w,\eta-\mu\rangle\neq 0\,\}.
$$
Since $E$ separates $K$, the open sets $U_w,w\in E$ cover
the complement of the diagonal in $K\times K$.
Since this complement is 
open in the separable  metrizable set $K\times K$,
 we can extract a countable
subcovering from this covering.
So we have a sequence 
$U_{w_k}$,
with $w_k\in E$, which covers the complement of the diagonal
in $K\times K$.
This amounts to say that the sequence $w_k$ 
separates $K$.
Defining $T_m=(w_1,\ldots,w_m)$, we have to prove that
\eqref{diam} holds for 
$m$ large enough.
Otherwise, we would have 
two sequences $\eta_m$ and $\mu_m$ in $K$
such that 
$$T_m(\mu_m)=T_m(\eta_m)\quad \text{ and }\quad
d(\mu_m,\eta_m)\ge\e.
$$
By extracting a subsequence, we can assume
that the sequences $\mu_m$ and $\eta_m$ have different limits
$\mu$ and $\eta$, which satisfy $d(\eta, \mu)\ge\e$.
Take $m$ large enough, so that 
$T_m(\eta)\neq T_m(\mu)$.
Such a value of $m$ exists because the linear forms $w_k$
separate $K$.
We have that 
$$
T_m(\mu_k)=T_m(\eta_k)\quad\text{ for }\quad k\geq m.
$$
Hence at the limit $T_m(\eta)=T_m(\mu)$. 
This is a contradiction.

\end{proof}

Define the function $F_m:A\times\re^m\to\re\cup\{+\infty\}$ as
$$
F_m(L,x):=\min_{\stackrel{\scriptstyle\mu\in H}
                  {T_m\circ \pi(\mu)=x}
                 }
(L-w)(\mu),
$$
when $x\in T_m(\pi(H))$ and $F_m(L,x)=+\infty$ if
$x\in\re^m\setminus T_m(\pi(H))$.
For $y=(y_1,\ldots,y_m)\in \Rm^m$,
let
$$
M_m(L,y) :=\argmin_{x\in\re^m}\big[F_m(L,x)-y\cdot x\big]\subset \re^m
$$
be the set of points which minimize
the function $x\lmto F_m(L,x)-y\cdot x$.
We have that
$$
M_K\big(L-w-y_1w_1-\cdots-y_mw_m\big)
\subset T_m^{-1}(M_m(L,y)).
$$
Let
$$
\cO_m(A,\dim A):=\{\;y\in\re^m\;|\;
\fa L\in A\;:\; \dim M_m(L,y)\le \dim A\;\}.
$$
From Lemma~\ref{map} it follows that
$$
y\in\cO_m(A,\dim A)\quad \then \quad
w+y_1w_1+\cdots+y_mw_m\in\cO(A,\e,\dim A).
$$
Therefore, in order to prove that $w$ is in the closure
of $\cO(A,\e,\dim A)$, it is enough to prove that
$0$ is in the closure of  $\cO_m(A,\dim A)$,
which follows from the next proposition.

\bigskip

\begin{prop}
The set $\cO_m(A,\dim A)$ 
 is dense in $\Rm^m$.
\end{prop}
\begin{proof}
Consider the Legendre transform of $F_m$ 
with respect to the second variable,
\begin{align*}
G_m(L,y)&=\max_{x\in \Rm^m} y\cdot x-F_m(L,x)
\\
&=
\max_{\mu\in H}
\langle w+y_1w_1+\cdots+y_mw_m,\mu\rangle - L(\mu).
\end{align*}
It follows from this second expression that the function
$G_m$ is convex and finite-valued, hence continuous
on $A\times \Rm^m$.

Consider the set $\tilde \Sigma \subset A\times \Rm^m$
of points $(L,y)$ such that $\dim \partial G_m(L,y)\geq \dim A+1$,
where $\partial G_m$ is the subdifferential of $G_m$.
It is known, see the appendix,
that this set has Hausdorff dimension at most 
$$
(m+\dim A)-(\dim A+1)=m-1.
$$
Consequently, the projection $\Sigma$ of the set $\tilde \Sigma$
on the second factor $\Rm^m$
also has Hausdorff dimension at most $m-1$.
Therefore, the complement of $\Sigma$ is dense in $\Rm^m$.
So it is enough to prove that
$$
y\notin\Si \quad\then\quad \fa L\in A\;:\quad
\dim M_m(L,y)\leq \dim A.
$$
% So it is enough to prove that, if $y$ is not in $\Sigma$,
% then the inequality
% $\dim M_m(L,y)\leq \dim A$
%  holds for all $L\in A$.
 Since we know by definition of $\Sigma$
 that $\dim \partial G_m(L,y)\leq \dim A$,
 it is enough to observe that
 $$
 \dim M_m(L,y)\leq \dim \partial G_m(L,y).
 $$
The last inequality follows from the fact that
the set $M_m(L,y)$
is the subdifferential of the convex function
$$
\Rm^m\ni z\lmto G_m(L,z)
$$
at the point $y$.

\end{proof}

 \section{Application to Lagrangian dynamics}

 Let $C$ be the set of continuous functions $f:\TTM\to\re$ with linear 
 growth, i.e.
 $$
 \lV f\rV_\ell:=\sup_{(t,\th)\in \TTM}\frac{|f(t,\th)|}{1+|\th|}<+\infty,
 $$
 endowed with the norm $\lV\cdot\rV_\ell$.

 We apply Theorem~\ref{abstract} to the following setting:

 \begin{itemize}
 \item $F=C^*$ is the vector space of continuous linear functionals
       $\mu:C\to\re$ provided with the weak-$\star$
       topology. Recall that  
       $$
       \lim_n\mu_n=\mu\quad\Longleftrightarrow\quad
       \lim_n\mu_n(f)=\mu(f), \quad\fa f\in C.
       $$
 \item 
 %correction
 $E=C^\infty(\T\times M,\re)$ provided with the $C^{\infty}$ topology.
 \item $G$ is the vector space of finite Borel signed measures on
  $\T\times M$,
 or equivalently the set of continuous linear forms on
 $C^0(\TM,\Rm)$,  provided with the  weak-$\star$ topology.
 \item The pairing $E\times G\to\re$ is given by integration:
       $$
       \lan u,\nu\ran =\int_{\TM}u\;d\nu.
       $$

  \item
The continuous linear map $\pi:F\lto G$  is induced by the projection
$\TTM\lto \T\times M $.

 \item The compact $K\subset G$ is the set of Borel probability measures
       on $\TM$, provided with the weak-$\star$ topology. Observe that
       $K$ is separated by $E$.      
     
 \item The compact 
 %correction
 $H_n\subset F$ is the set of holonomic
       probability measures which are supported on  
       $$
       B_n:=\{(t,\th)\in \TTM\;|\;|\th|\le n\;\}.
       $$
\end{itemize}

 Holonomic probabilities are defined as follows:   
 Given a $C^1$ curve $\ga:\re\to M$ of period $T\in\na$ define the
 element  $\mu_\ga$ of $F$ by
 $$
 \lan f,\mu_\ga \ran=\frac 1T\int_0^Tf(s,\ga(s),\dga(s))\;ds
 $$
 for each $f\in C$. Let 
%correction 
 $$
 \Ga:=\{\;\mu_\ga\;|\;\ga\in C^1(\re,M)\text{ is periodic
 of integral period }\}\subset F.
 $$
 The set $\cH$ of holonomic probabilities is the closure of $\Ga$ in
 $F$. One can see that $\cH$ is convex (cf.~Ma\~n\'e~\cite[prop. 1.1(a)]{Ma6}).
 The elements $\mu$ of $\cH$ satisfy $\lan 1, \mu\ran=1$
 therefore we have $\pi(\cH)\subset K$.
 
 Note that each Tonelli Lagrangian $L$ gives rise
 to an element of $H_n^*$.

 Let $\fM(L)$ be the set of minimizing measures for $L$ and let
 $\supp\fM(L)$ be the union of their supports.
 Recalling that we have defined $M_{H_n}(L)$
 as the set of measures $\mu\in H_n$ which minimize 
 the action $\int L\,d\mu$ on $H_n$, we have:

 \begin{lem}\label{measures}
 If $L$ is a Tonelli lagrangian then there exists $n\in\na$ such that 
 $$
 \dim\pi(M_{H_n}(L))=\dim\fM(L).
 $$
 \end{lem}
 
 \begin{proof}\quad
 
 Birkhoff theorem
 implies that $\fM(L)\subset \cH$ (cf.~Ma\~n\'e~\cite[prop.~1.1.(b)]
 {Ma6}).
 In~\cite[Prop.~4, p.~185]{Mat5} Mather proves that $\supp\fM(L)$ 
 is compact, therefore
 $
 \fM(L)\subset H_n
 $
for some $n\in\na$.
 
  In~\cite[\S 1]{Ma6} Ma\~n\'e proves that minimizing measures 
 are also all the minimizers of the action functional
 $
 A_L(\mu)=\int L\;d\mu
 $
 on the set of holonomic measures, therefore 
 $
 \fM(L)=M_{H_n}(L)
 $
 for some $n\in\na$.

 In~\cite[Th.~2, p.~186]{Mat5} Mather proves that  the restriction 
 \;$\supp\fM(L)\to M$\; of the projection $TM\to M$ is injective. 
 Therefore the linear map $\pi:\fM(L)\to G$ is injective,
 so that $\dim \pi (  M_{H_n}(L))=\dim\pi(\fM(L))=\dim\fM(L).$
 \end{proof}

 \noindent{\bf Proof of Theorem \ref{main}.}
 
 Given $n\in\na$ apply Theorem~\ref{abstract} and obtain a 
 residual subset $\cO_n(A)\subset E$ such that
 $$
 L\in A,\quad  u\in \cO_n(A)\;\Longrightarrow
 \qquad \dim\pi(M_{H_n}(L-u))\le\dim A.
 $$
 Let 
 $\cO(A)=\cap_n\cO_n(A)$. By the Baire property $\cO(A)$ 
 is residual. We have that 
 $$
 L\in A,\quad  u\in \cO(A),\quad  n\in \na\;\Longrightarrow
 \qquad \dim\pi(M_{H_n}(L-u))\le\dim A .
 $$
 Then by Lemma~\ref{measures}, $\dim\fM(L-u)\le\dim A$ for
  all $L\in A$ and all $u\in \cO(A)$.
 \qed

\begin{appendix}

\section{Convex Functions}

Given a convex function $f:\re^n\to\re$ and $x\in\re^n$, define its subdifferential as
$$
\partial f(x):=\{\,\ell:\re^n\to\re\;\text{ linear }|\; f(y)\ge f(x)+\ell(y-x)\,,\;\forall y\in\re^n\;\}.
$$
Then the sets $\partial f(x)\subset\re^n$ are convex. If $k\in\na$, let
$$
\Si_k(f):=\{\,x\in\re^n\;|\;\dim \partial f(x)\ge k\;\}.
$$
The following result is standard. 
\begin{Proposition}\label{cxprop}
If $f:\re^n\to\re$ is a convex function then for all $0\le k\le n$ the Hausdorff dimension $HD(\Si_k(f))\le n-k$.
\end{Proposition}
We recall here an elegant proof due to Ambrosio and Alberti,
see \cite{AA}. Note that much more can be said on the structure
of $\Sigma_k$, see \cite{alberti,Z} for example.

By adding $\lv x\rv^2$ if necessary
(which does not change $\Si_k$)
 we can assume that $f$ is superlinear and that
\begin{equation}\label{ellineq}
f(y)\ge f(x)+\ell(y-x)+\tfrac 12\,\lv y-x\rv^2\quad
\forall x,y\in\re^n,\quad\forall\ell\in\partial f(x).
\end{equation}

\begin{Lemma}
 \qquad $\ell\in\partial f(x)$, \quad $\ell'\in\partial f(x')$\quad $\then$ \quad $\lv x-x'\rv\le \lV\ell-\ell'\rV$.
\end{Lemma}

\begin{proof}

From inequality~\eqref{ellineq} we have that 
\begin{align*}
f(x')&\ge f(x)+\ell(x'-x)\;+\,\tfrac 12\,\lv x'-x\rv^2,\\
f(x)&\ge f(x')+\ell'(x-x')+\tfrac 12\lv x-x'\rv^2.
\end{align*}
Then 
\begin{gather}
0\ge (\ell'-\ell)(x-x')+\lv x-x'\rv^2 \\
\lV \ell-\ell'\rV \,\lv x-x'\rv\ge (\ell-\ell')(x-x')\ge\lv x-x'\rv^2.
\end{gather}
Therefore $\lV\ell-\ell'\rV\ge\lv x-x'\rv$.
\end{proof}

Since $f$ is superlinear, the subdifferential $\partial f$ 
is surjective and we have:
\begin{Corollary}
There exists a Lipschitz function $F:\re^n\to\re^n$ such that
 $$
 \ell\in\partial f(x) \quad \then \quad x=F(\ell).
 $$ 
\end{Corollary}

\noindent{\bf Proof of Proposition~\ref{cxprop}:}

Let $A_k$ be a set with $HD(A_k)=n-k$ such that $A_k$ intersects any convex subset of dimension 
$k$. For example 
$$
A_k=\{ x\in\re^n\, |\, x \text{ has at least $k$ rational coordinates }\}.
$$
Observe that 
$$
x\in\Si_k\quad \then \quad \partial f(x) \text{ intersects } A_k 
\quad \then \quad x\in F(A_k). 
$$
Therefore $\Si_k\subset F(A_k)$.
Since $F$ is Lipschitz,  we have that $HD(\Si_k) \le HD(A_k)=n-k$.
\qed
\end{appendix}

 \nocite{Mat7}
 \nocite{Ma3}

    \bibliographystyle{amsplain}
    \bibliography{biblio}

 %\def\cprime{$'$} \def\cprime{$'$} \def\cprime{$'$} \def\cprime{$'$}
 %\providecommand{\bysame}{\leavevmode\hbox to3em{\hrulefill}\thinspace}
 %\providecommand{\MR}{\relax\ifhmode\unskip\space\fi MR }
% \providecommand{\MRhref}[2]{%
%   \href{http://www.ams.org/mathscinet-getitem?mr=#1}{#2}
% }
% \providecommand{\href}[2]{#2}

   \end{document}